\documentclass{amsart}
\usepackage{all2015}
\newcommand{\Fin}{\mathbf{Fin}}
\newcommand{\OS}{\mathbf{OS}}
\newcommand{\op}{\mathrm{op}}
\newcommand{\bC}{\mathbf{C}}
\newcommand{\cV}{\mathcal{V}}
\newcommand{\cX}{\mathcal{X}}
\newcommand{\cQ}{\mathcal{Q}}
\renewcommand{\phi}{\varphi}
\newcommand{\maxcl}{\operatorname{mcl}}

\usepackage{pxfonts}

\begin{document}

\title{Markov random fields and iterated toric fibre products}
\author{Jan Draisma and Florian M.~Oosterhof}

\address[Jan Draisma]{Mathematical Institute, University of Bern,
Sidlerstrasse 5, 3012 Bern, Switzerland; and Eindhoven
University of Technology, The Netherlands}
\email{jan.draisma@math.unibe.ch}

\address[Florian M.~Oosterhof]{Department of Mathematics and
Computer Science, Eindhoven University of Technology,
P.O.~Box 513, 5600 MB Eindhoven, The Netherlands}
\email{f.m.oosterhof@tue.nl}

\maketitle

\vspace{-.5cm}

\begin{abstract}
We prove that iterated toric fibre products from a finite collection of
toric varieties are defined by binomials of uniformly bounded degree.
This implies that Markov random fields built up from a finite collection
of finite graphs have uniformly bounded Markov degree.
\end{abstract}

\section{Introduction and main results} \label{sec:Intro}

The notion of {\em toric fibre product} of two toric varieties goes back
to \cite{Sullivant07}. It is of relevance in algebraic statistics since
it captures algebraically the Markov random field on a graph obtained
by glueing two graphs along a common subgraph; see \cite{Rauh14} and
also below. In \cite{Sullivant07,Rauh14,Rauh14b} it is proved that under
certain conditions, one can explicitly construct a Markov basis for the
large Markov random field from bases for the components.  For
related results see \cite{Shibuta12,Engstrom14,Kahle14}.

However, these conditions are not always satisfied.  
Nevertheless, in
\cite[Conjecture 56]{Rauh14} the hope was raised that when building
larger graphs by glueing copies from a finite collection of graphs along a common subgraph,
there might be a uniform upper bound on the Markov degree of the models
thus constructed, independent of how many copies of each graph are
used. A special case of this conjecture was proved in the same paper
\cite[Theorem 54]{Rauh14}. We prove the conjecture in general,
and along the way we link it to recent work \cite{Sam14} in {\em
representation stability}. Indeed, an important point we would like
to make, apart from proving said conjecture, is that algebraic statistics
is a natural source of problems in {\em asymptotic algebra}, to which ideas 
from representation stability apply. Our main theorems are reminiscent of 
Sam's recent stabilisation theorems on equations and higher syzygies for 
secant varieties of Veronese embeddings \cite{Sam15b,Sam16}. 

\subsection*{Markov random fields}
Let $G=(N,E)$ be a finite, undirected, simple graph and for
each node $j \in N$ let $X_j$ be a random variable taking
values in the finite set $[d_j]:=\{1,\ldots,d_j\}$. A joint
probability distribution on $(X_j)_{j\in N}$ is said to
satisfy the {\em local Markov properties} imposed by the
graph if for any two non-neighbours $j,k \in N$ the variables $X_j$
and $X_k$ are conditionally independent given $\{X_l \mid
\{j,l\} \in E\}$.

On the other hand, a joint probability distribution $f$ on the $X_j$ is said to {\em factorise according to $G$} if for each maximal clique $C$ of $G$ and configuration $\alpha \in \prod_{j \in C}[d_j]$ of the random variables labelled by $C$ there exists an interaction parameter $\theta^C_\alpha$ such that for each configuration $\beta \in \prod_{j \in N}[d_j]$ of all random variables of $G$:
\[
f(\beta) = \prod_{C\in\maxcl(G)}\theta^C_{\beta|_C}
\]
where $\maxcl(G)$ is the set of maximal cliques of $G$, and $\beta|_C$ is the restriction of $\beta$ to $C$.

These two notions are connected by the Hammersley-Clifford theorem, which says that a positive joint probability distribution on $G$ factorises according to $G$ if and only if it satisfies the Markov properties; see \cite{Hammersley71} or \cite[Theorem 3.9]{Lauritzen98}.

The set of all positive joint probability distributions on $G$ that satisfy the Markov properties is therefore a subset of the image of the following map:
\[
\phi_G: \CC^{\prod_{C\in\maxcl(G)}\prod_{j\in
C}[d_j]}\to\CC^{\prod_{j\in N}[d_j]},\qquad (\theta^C_\alpha)_{C,\alpha}\mapsto \left(\prod_{C\in\maxcl(G)}\theta^C_{\beta|_C}\right)_{\beta}
\]
It is the ideal $I_G$ of polynomials vanishing on $\im \phi_G$ that is of 
interest in algebraic statistics. Since the components of
$\phi_G$ are monomials, $I_G$ is generated by finitely many
binomials (differences of two monomials) in the standard
coordinates on $\CC^{\prod_{j \in N} [d_j]}$, and any
finite generating set of binomials can be used to set up a
Markov chain for testing whether given observations of the
variables $X_j$ are compatible with the assumption that
their joint distribution factorises according to the graph
$G$ \cite{Diaconis98}. The zero locus of $I_G$ is often
called the {\em graphical model} of $G$.

Now suppose we have graphs $G_1,\ldots,G_s$ with node sets $N_1,\ldots,N_s$, that $N_i \cap N_k$ equals a fixed set $N_0$ for all $i \neq k$ in $[s]$, and that the graph induced on $N_0$ by each $G_i$ is equal to a fixed graph $H$. Moreover, for each $j \in \bigcup_i N_i$ fix a number $d_j$ of states.  
We can then glue copies of the $G_i$ along their common subgraph $H$, by which we mean first taking disjoint copies of the $G_i$ and then identifying the nodes labelled by a fixed $j \in N_0$ across all copies. 
For $a_1,\ldots,a_s\in\ZZ_{\geq 0}$, we denote the graph obtained by glueing $a_i$ copies of graph $G_i, i \in [s]$ by $\sum_H^{a_1}G_1+_H\cdots+_H\sum_H^{a_s}G_s$:
\begin{center}
\includegraphics[width=\textwidth]{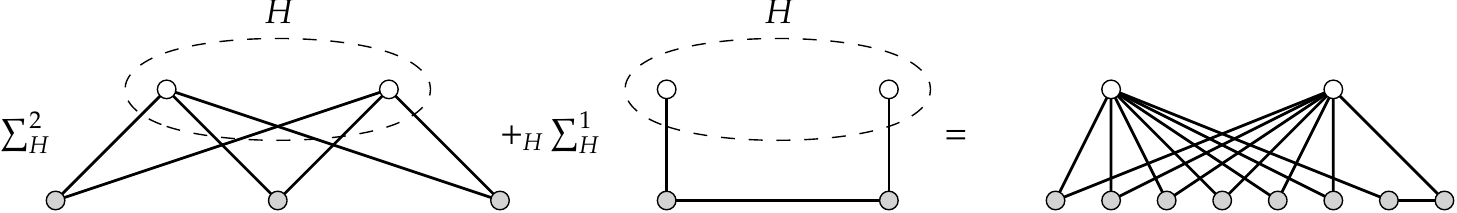}
\end{center}
\begin{thm} \label{cor:Markov}
	Let $G_1,\ldots,G_s$ be graphs with a common subgraph $H$ and a number of states associated to each node. Then there exists a uniform bound $D\in\ZZz$ such that for all multiplicities $a_1,\ldots,a_s$, the ideal $I_G$ of $G = \sum_H^{a_1}G_1+_H\cdots+_H\sum_H^{a_s}G_s$ is generated by binomials of degree at most $D$.
\end{thm}

Our proof shows that one needs only finitely many combinatorial
types of binomials, independent of $a_1,\ldots,a_s$, to generate
$I_G$. This result is similar in flavour to the Independent Set
Theorem from~\cite{Hillar09}, where the graph is fixed but the
$d_j$ vary. Interestingly, the underlying categories responsible
for these two stabilisation phenomena are opposite to each other; see
Remark~\ref{re:Independent}.

\begin{ex}
In \cite{Rauh14b} it is proved that the ideal $I_G$ for the complete
bipartite graph $G=K_{3,N}$, with two states for each of the random
variables, is generated in degree at most 12 for all $N$. The graph 
$G$ is obtained by glueing $N$ copies of $K_{3,1}$ along their common 
subgraph consisting of $3$ nodes without any edges. 
\end{ex}

We derive Theorem~\ref{cor:Markov} from a general stabilisation 
result on {\em toric fibre products}, which we introduce next.

\subsection*{Toric fibre products}

Fix a ground field $K$, let $r$ be a natural number, and let
$U_1,\ldots,U_r$, $V_1,\ldots,V_r$ be finite-dimensional vector spaces
over $K$. 

Define a bilinear operation
\begin{equation} \label{eq:Toric}
\prod_j U_j \times \prod_j V_j \to \prod_j (U_j \otimes
V_j),\quad (u,v)\mapsto u*v:=(u_1 \otimes v_1,\ldots,u_r \otimes
v_r). 
\end{equation}

\begin{de}
The {\em toric fibre product} $X*Y$ of Zariski-closed subsets $X \subseteq
\prod_j U_j$ and $Y \subseteq \prod_j V_j$ equals the Zariski-closure
of the set $\{u*v \mid u \in X, v \in Y\}$. \hfill $\clubsuit$
\end{de}

\begin{re} \label{re:Ideal}
In~\cite{Sullivant07}, the toric fibre product is defined at the level of
ideals: if $(x^j_i)_i$ are coordinate functions on $U_j$ and $(y^j_k)_k$
are coordinate functions on $V_j$, then $(z^j_{i,k}:=x^j_i \otimes
y^j_k)_{i,k}$ are coordinate functions on $U_j \otimes V_j$. The ring
homomorphism of coordinate rings
\[ K[(z^j_{i,k})_{j,i,k}]=
K\left[\prod_j (U_j \otimes V_j)\right] \to 
K\left[\prod_j U_j \times \prod_j V_j \right]
= K[(x^j_i)_{j,i}, (y^j_k)_{j,k}] 
\]
dual to~\eqref{eq:Toric} sends $z^j_{i,k}$ to $x^j_i \cdot y^j_k$. If
we compose this homomorphism with the projection modulo the ideal
of $X \times Y$, then the kernel of the composition is precisely
the toric fibre product of the ideals of $X$ and $Y$ as introduced
in~\cite{Sullivant07}. In that paper, multigradings play a crucial role
for {\em computing} toric fibre products of ideals, but do
not affect the {\em definition} of toric fibre products.
\hfill
$\clubsuit$
\end{re}

The product $*$ is associative and commutative up to reordering tensor factors. We can iterate this construction and form products like $X^{*2} * Y * Z^{*3}$, where $Z$ also lives in a product of $r$ vector spaces $W_j$. This variety lives in $\prod_j (U_j^{\otimes 2}) \otimes V_j
\otimes (W_j^{\otimes 3})$.

We will not be taking toric fibre products of general varieties,
but rather {\em Hadamard-stable} ones. For this, we have to choose
coordinates on each $U_j$, so that $U_j=K^{d_j}$.

\begin{de}
On $K^d$ the Hadamard product is defined as $(a_1,\ldots,a_d) \bigcirc
(b_1,\ldots,b_d)=(a_1 b_1,\ldots,a_d b_d)$. On $U:=U_1 \times \cdots
\times U_r$, where $U_j=K^{d_j}$, it is defined component-wise. A set $X
\subseteq \prod_j U_j$ is called {\em Hadamard-stable} if $X$ contains
the all-one vector $\one_U$ (the unit element of $\bigcirc$) and if
moreover $x \bigcirc z \in X$ for all $x,z \in X$. \hfill $\clubsuit$
\end{de}

\begin{re}
By~\cite[Remark after Proposition 2.3]{Eisenbud96}, $X$ is Hadamard-stable
if and only if its ideal is generated by differences of two monomials.
\end{re}

In particular, the Zariski-closure in $U$ of a subtorus of
the $\sum_j d_j$-dimensional torus $\prod_j (K \setminus
\{0\})^{d_j}$ is Hadamard-stable. These are the toric
varieties from the abstract. 

Suppose that we also fix identifications $V_j=K^{d_j'}$ and a corresponding
Hadamard multiplication $\prod_j V_j \times \prod_j V_j \to \prod_j
V_j$. Equipping the spaces $U_j \otimes V_j$ with the natural coordinates
and corresponding Hadamard multiplication, the two operations just
defined satisfy $(u \bigcirc u')*(v \bigcirc
v')=(u*v)\bigcirc(u'*v')$ as well as $\one_{U_j \otimes
V_j}=\one_{U_j} \otimes \one_{V_j}$. Consequently, if both
$X \subseteq \prod_j U_j$ and $Y \subseteq \prod_j V_j$ are
Zariski-closed and Hadamard-stable, then so is their toric
fibre product $X*Y$. We can now formulate our second main result. 

\begin{thm} \label{thm:Main}
Let $s,r \in \ZZ_{\geq 0}.$ For each $i \in [s]$ and $j
\in [r]$ let $d_{ij} \in \ZZ_{\geq 0}$ and set $V_{ij}:=K^{d_{ij}}$. For
each $i \in [s]$, let $X_i \subseteq \prod_j V_{ij}$ be a Hadamard-stable
Zariski-closed subset. Then there exists a uniform bound $D
\in \ZZ_{\geq 0}$ such that for any exponents
$a_1,\ldots,a_s$ the ideal of $X_1^{*a_1} * \cdots *
X_s^{*a_s} \subseteq \prod_{j=1}^r \bigotimes_{i=1}^s
V_{ij}^{\otimes a_i}$ is generated by polynomials of degree
at most $D$.
\end{thm}

\begin{re}
A straightforward generalisation of this theorem also holds, where
each $X_i$ is a closed sub-scheme given by some ideal $J_i$ in the
coordinate ring of $\prod_j V_{ij}$. Hadamard-stable then says that
that the pull-back of $J_i$ under the Hadamard product lies in
the ideal of $X_i \times X_i$, and the toric fibre product is as in
Remark~\ref{re:Ideal}. Since this generality would slightly obscure
our arguments, we have decided to present explicitly the version with
Zariski-closed subsets---see also Remark~\ref{re:Scheme}.

Also, the theorem remains valid if we remove the condition that the
$X_i$ contain the all-one vector, but require only that they be closed
under Hadamard-multiplication; see Remark~\ref{re:NoOne}.
\hfill $\clubsuit$
\end{re}

\subsection*{Organisation of this paper}
The remainder of this paper is organised as follows. In
Section~\ref{sec:FS} we introduce the categories of (affine)
$\Fin$-varieties and, dually, $\Fin^\op$-algebras. The point is that,
as we will see in Section~\ref{sec:Proof}, the iterated toric fibre
products together form such a $\Fin$-variety (or rather a
$\Fin^s$-variety,
where $\Fin^s$ is the product category of $s$ copies of $\Fin$).  

Indeed,
they sit naturally in a Cartesian product of copies of the
$\Fin$-variety
of rank-one tensors, which, as we prove in Section~\ref{sec:RankOne},
is Noetherian (Theorem~\ref{thm:RankOne}). This Noetherianity result
is of a similar flavour as the recent result from \cite{Sam14}
(see also \cite[Proposition 7.5]{Draisma11d} which follows the same
proof strategy) that any finitely generated $\Fin^\op$-{\em module}
is Noetherian; this result played a crucial role in a proof of
the {\em Artinian conjecture}. However, our Noetherianity result
concerns certain $\Fin^\op$-{\em algebras} rather than modules,
and is more complicated. Finally, in 
Section~\ref{sec:Proof} we first prove
Theorem~\ref{thm:Main} and then derive 
Theorem~\ref{cor:Markov} as a corollary.

\subsection*{Acknowledgements}

Both authors are (partially) supported by Draisma's {\em Vici} grant {\em
Stabilisation in Algebra and Geometry} from the Netherlands Organisation
for Scientific Research (NWO).  The current paper is partly based on the
second author's Master's thesis at Eindhoven University of Technology;
see \cite{Oosterhof16}. We thank Johannes Rauh and Seth Sullivant for
fruitful discussions at {\em Prague Stochastics}, August 2014.

\section{Affine $\Fin$-varieties and $\Fin^\op$-algebras} \label{sec:FS}

The category $\Fin$ has as objects all finite sets and as morphisms
all maps. Its opposite category is denoted $\Fin^\op$. When $\bC$
is another category whose objects are called {\em somethings}, then a
{\em $\Fin$-something} is a covariant functor from $\Fin$ to $\bC$
and a {\em $\Fin^\op$-something} is a contravariant functor from $\Fin$
to $\bC$. The $\bC$-homomorphism associated to a $\Fin$-morphism $\pi$
is denoted $\pi_*$ (the {\em push-forward of $\pi$}) in the covariant
case and $\pi^*$ (the {\em pull-back of $\pi$}) in the contravariant
case. 

More generally, we can replace $\Fin$ by any category $\mathbf{D}$.
The $\mathbf{D}$-somethings themselves form a category, in which
morphisms are natural transformations. In our paper,
$\mathbf{D}$ is always closely related to $\Fin$ or to $\Fin^s$, the $s$-fold product of $\Fin$.

Here are three instances of $\Fin$ or $\Fin^\op$-somethings crucial to our paper.

\begin{ex} \label{ex:Fin}
Fix $n \in \ZZ_{\geq 0}$. Then the functor $S \mapsto [n]^S$ is a
$\Fin^\op$-set, which to $\pi \in \Hom_{\Fin}(S,T)$ associates the map
$\pi^*:[n]^T \to [n]^S, \alpha \mapsto \alpha \pi$, the composition of $\alpha$ and $\pi$.

Building on this example, the functor $A_n:S \mapsto K[x_\alpha \mid
\alpha \in [n]^S]$, a polynomial ring in variables labelled by $[n]^S$,
is a $\Fin^\op$-$K$-algebra, which associates to $\pi$ the $K$-algebra 
homomorphism $\pi^*: A(T) \to A(S),\quad x_\alpha \mapsto x_{\alpha \pi}$.

Third, we define an affine $\Fin$-$K$-vector space $Q_n$ by $Q_n: S \mapsto
(K^n)^{\otimes S}$, the space of $n \times \cdots \times n$-tensors
with factors labelled by $S$, which sends $\pi \in \Hom_\Fin(S,T)$
to the linear morphism $\pi_*:(K^n)^{\otimes S} \to (K^n)^{\otimes T}$
determined by
\[\pi_*: \otimes_{i \in S} v_i \mapsto 
\otimes_{j \in T}
\left(\bigcirc_{i
\in \pi^{-1}(j)} v_i\right), \]
where $\bigcirc$ is the Hadamard product in $K^n$. We follow the natural
convention that an empty Hadamard product equals the all-one vector
$\one \in K^n$; in particular, this holds in the previous formula for
all $j \in T$ that are not in the image of $\pi$.

The ring $A_n$ and the space $Q_n$ are related as follows: $Q_n(S)$ has a
basis consisting of vectors $e_\alpha:= \otimes_{i \in S} e_{\alpha_i},\
\alpha \in [n]^S$, where $e_1,\ldots,e_n$ is the standard basis of $K^n$; 
$A_n(S)$ is the coordinate ring of $Q_n(S)$ generated by the dual basis
$(x_\alpha)_{\alpha \in [n]^S}$; and for $\pi  \in \Hom_\Fin(S,T)$ the pullback
$\pi^*:A_n(T) \to A_n(S)$ is the homomorphism of $K$-algebras dual to the
linear map $\pi_*:Q_n(S) \to Q_n(T)$. Indeed, this is verified by the following 
computation for $\alpha \in [n]^T$:
\[ x_\alpha (\pi_* \otimes_{i \in S} v_i)
= x_\alpha (\otimes_{j \in T} \left(\bigcirc_{i \in \pi^{-1}(j)}v_i \right))
= \prod_{j \in T} \left(\bigcirc_{i \in \pi^{-1}(j)}v_i \right)_{\alpha_j}
= \prod_{i \in S} (v_i)_{\alpha_{\pi(i)}} 
= \pi^*(x_\alpha) (\otimes_{i \in S} v_i). \]
This is used in Section~\ref{sec:RankOne}.
$\clubsuit$
\end{ex}

In general, by {\em algebra} we shall mean an associative, commutative
$K$-algebra with $1$, and homomorphisms are required to preserve $1$.
So a $\Fin^\op$-algebra $B$ assigns to each finite set $S$ an algebra
and to each map $\pi:S \to T$ an algebra homomorphism
$\pi^*:B(T)
\to B(S)$. An {\em ideal} in $B$ is a $\Fin^\op$-subset $I$ of $B$
(i.e., $I(S)$ is a subset of $B(S)$ for each finite set $S$
and $\pi^*$ maps $I(T)$ into $I(S)$) such that
each $I(S)$ is an ideal in $B(S)$; then $S \mapsto B(S)/I(S)$ is again
a $\Fin^\op$-algebra, the {\em quotient} $B/I$ of $B$ by
$I$.

Given a $\Fin^\op$-algebra $B$, finite sets $S_j$ for $j$ in some index
set $J$, and an element $b_j \in B(S_j)$ for each $j$, there is a unique
smallest ideal $I$ in $B$ such that each $I(S_j)$ contains $b_j$. This ideal is constructed as:
\[
I(S) = \left(\pi^*(b_j) \,\middle|\, j\in J, \pi\in\Hom(S,S_j)\right)
\]
This is the ideal {\em generated} by the $b_j$. A $\Fin^\op$-algebra is called {\em Noetherian} if each ideal $I$ in it is generated by finitely many
elements in various $I(S_j)$, i.e., $J$ can be taken finite.

\begin{ex}
The $\Fin^\op$-algebra $A_1$ is Noetherian. Indeed, $A_1(S)$ is the polynomial
ring $K[t]$ in a single variable $t$ for all $S$, and the homomorphism
$A_1(T) \to A_1(S)$ is the identity $K[t] \to K[t]$. So Noetherianity
follows from Noetherianity of the algebra $K[t]$.

For $n \geq 2$ the $\Fin^\op$-algebra $A_n$ is {\em not}
Noetherian. For instance, consider the monomials
\[ u_2:=x_{21}x_{12}  \in A_n([2]), \ 
u_3:=x_{211}x_{121}x_{112}  \in A_n([3]), \ 
u_4:=x_{2111}x_{1211}x_{1121}x_{1112} \in A_n([4]), 
\]
and so on. For any map $\pi:[k] \to [l]$ with $k>l$, by the pigeon hole principle
there are two indices $i,j \in [k]$ such that $\pi(i)=\pi(j)=:m \in [l]$.
Then $\pi^* x_{1 \cdots 1 2 1 \cdots 1}$, where the $2$ is in the $m$-th
position, is a variable with at least two indices equal to $2$. Since $u_k$ 
contains no such variable, 
$\pi^* u_l$ does not divide $u_k$. So $u_2,u_3,\ldots$ generates a
non-finitely generated monomial $\Fin^\op$-ideal in $A_n$.  (On the other
hand, for each $d$ the piece of $A_n$ of homogeneous polynomials of
degree at most $d$ {\em is} Noetherian as a $\Fin^\op$-module, see
\cite[Proposition 7.5]{Draisma11d}.) \hfill $\clubsuit$
\end{ex}

We shall see in the following section that certain interesting
quotients of each $A_n$ {\em are} Noetherian.

\section{Rank-one tensors form a Noetherian $\Fin$-variety} \label{sec:RankOne}

Let $Q^{\leq 1}_n(S)$ be the variety of {\em rank-one tensors}, i.e.,
those of the form $\otimes_{i \in S} v_i$ for vectors $v_i \in K^n$. We
claim that this defines a Zariski-closed $\Fin$-subvariety $Q^{\leq 1}_n$
of $Q_n$.

For this, we must verify that for a map $\pi:S \to T$ the
map $Q_n(S) \to Q_n(T)$ dual to the algebra homomorphism $A_n(T) \to
A_n(S)$ sends $Q_n^{\leq 1}(S)$ into $Q_n^{\leq 1}(T)$. And indeed,
in Example~\ref{ex:Fin} we have seen that this map sends
\[ \otimes_{i \in S} v_i \mapsto \otimes_{j \in T}
\left(\bigcirc_{i
\in \pi^{-1}(j)} v_i \right).\] 
It is well known that (if $K$ is infinite) the ideal in $A_n(S)$ of
$Q_n^{\leq 1}(S)$ equals the ideal $I_n(S)$ generated by all binomials
constructed as follows. Partition $S$ into two parts $S_1,S_2$, let
$\alpha_i,\beta_i \in [n]^{S_i}$ and write $\alpha_1||\alpha_2$ for
the element of $[n]^S$ which equals $\alpha_i$ on $S_i$. Then we have
the binomial
\begin{equation} \label{eq:Det}
x_{\alpha_1||\alpha_2}x_{\beta_1||\beta_2}-x_{\alpha_1||\beta_2}x_{\beta_1||\alpha_2}
\in I_n(S), 
\end{equation}
and $I_n(S)$ is the ideal generated by these for all partitions and all
$\alpha_1,\alpha_2,\beta_1,\beta_2$. The functor $S \mapsto I_n(S)$ is
an ideal in the $\Fin^\op$-algebra $A_n$; for infinite $K$ this follows
from the computation above, and for arbitrary $K$ it follows since the
binomials above are mapped to binomials by pull-backs of maps $T \to S$
in $\Fin$.  Moreover, $I_n$ is finitely generated (see also \cite[Lemma
7.4]{Draisma11d}):

\begin{lm}
The ideal $I_n$ in the $\Fin^\op$-algebra $A_n$ is finitely generated.
\end{lm}

\begin{proof}
In the determinantal equation~\eqref{eq:Det}, if there exist distinct $j,l
\in S_1$ such that $\alpha_1(j)=\alpha_2(l)$ and $\beta_1(j)=\beta_2(l)$,
then the equation comes from an equation in $I_n(S \setminus \{j\})$ via
the map $S \to S \setminus \{j\}$ that is the identity on $S \setminus
\{j\}$ and maps $j \to l$. By the pigeon hole principle this happens
when $|S_1|>n^2$.  Similarly for $|S_2|>n^2$. Hence $I_n$ is certainly
generated by $I_n([2n^2-1])$.
\end{proof}

The main result in this section is the following.

\begin{thm} \label{thm:RankOne}
For each $n \in \ZZ_{\geq 0}$ the coordinate ring $A_n/I_n$ of the $\Fin$-variety
$Q_n^{\leq 1}$ of rank-one tensors is a Noetherian $\Fin^\op$-algebra.
\end{thm}

Our proof follows the general technique from \cite{Sam14}, namely, to pass
to a suitable category close to $\Fin$ that allows for a Gr\"obner basis
argument. However, the relevant well-partial-orderedness proved below
is new and quite subtle.  We use the category $\OS$ from \cite{Sam14}
(also implicit in \cite[Section 7]{Draisma11d}) defined as follows.

\begin{de}
The objects of the category $\OS$ (``ordered-surjective'') are all finite
sets equipped with a linear order and the morphisms $\pi: S \to T$ are all
surjective maps with the additional property that the function $T \to S,\
j \mapsto \min \pi^{-1}(j)$ is strictly increasing. \hfill $\clubsuit$
\end{de} 

Any $\Fin$-algebra is also an $\OS$-algebra, and $\OS$-Noetherianity
implies $\Fin$-Noe\-the\-ria\-nity. So to prove Theorem~\ref{thm:RankOne}
we set out to prove the stronger statement that $A_n/I_n$ is, in fact,
$\OS$-Noetherian.

We get a more concrete grip on the $K$-algebra $A_n/I_n$ through the
following construction. Let $M_n$ denote the (Abelian) $\Fin^\op$-monoid
defined by
\[
M_n(S):=\left\{\alpha \in \ZZ_{\geq 0}^{[n] \times S} \mid \forall j,l \in S: \sum_{i=1}^n \alpha_{ij}=\sum_{i=1}^n \alpha_{il}
\right\},
\]
in which the multiplication is given by addition, and where the pull-back of a map $\pi:S \to T$ is the map $\pi^*: M_n(T) \to M_n(S)$ sending
$(\alpha_{ij})_{i \in [n],j \in T}$ to $(\alpha_{i\pi(j)})_{i \in
[n], j \in S}$. So elements of $M_n(S)$ are matrices with nonnegative
integral entries and constant column sum. Let $K M_n$ denote the
$\Fin$-algebra sending $S$ to the monoid $K$-algebra $K M_n(S)$. The
following proposition is a reformulation of a well-known fact.

\begin{prop} The $\Fin^\op$-algebra $A_n/I_n$ is isomorphic to the
$\Fin^\op$-algebra $K M_n$ (and the same is true when both are regarded as
$\OS^\op$-algebras). 
\end{prop}

\begin{proof}
For each finite set $S$, the $K$-algebra homomorphism $\Phi_S:A_n(S) \to
KM_n(S)$ that sends $x_\alpha, \alpha \in [n]^S$ to the $[n] \times S$
matrix in $M_n(S)$ with a $1$ at the positions $(\alpha_j,j),\ j \in S$
and zeroes elsewhere is surjective and has as kernel the ideal $I_n(S)$.
Moreover, if $\pi:S \to T$ is a morphism in $\Fin$, then we have $\Phi_T
\pi^*=\pi^* \Phi_S$, i.e., the $(\Phi_S)_S$ define a natural transformation.
\end{proof}

Choose any monomial order $>$ on $\ZZ_{\geq 0}^n$, i.e., a
well-order such that $a>b$ implies $a+c > b+c$ for every $a,b,c \in
\ZZ_{\geq 0}^n$. Then for each object $S$ in $\OS$ we define a linear
order $>$ on $M_n(S)$, as follows: $\alpha > \beta$ if $\alpha \neq
\beta$ and the smallest $j \in S$ with $\alpha_{.,j} \neq \beta_{.,j}$
(i.e., the $j$-th column of $\alpha$ is not equal to that of $\beta$)
satisfies $\alpha_{.,j}>\beta_{.,j}$ in the chosen monomial order on
$\ZZ_{\geq 0}^n$. A straightforward verification shows that this is
a monomial order on $M_n(S)$ (we call the elements of $M_n(S)$
monomials, even though $KM_n(S)$ is not a polynomial ring). Moreover,
for various $S$, these orders are interrelated as follows.

\begin{lm} \label{lm:Order}
For any $\pi \in \Hom_{\OS}(S,T)$ and $\alpha,\beta \in M_n(T)$, we have
$\alpha > \beta \Rightarrow \pi^* \alpha > \pi^* \beta$.
\end{lm}

\begin{proof}
If $j \in T$ is the smallest column index where $\alpha$ and $\beta$ differ,
then $\alpha':=\pi^* \alpha$ and $\beta':=\pi^* \beta$ differ in column
$l:=\min \pi^{-1}(j)$, where they equal $\alpha_{.,j}$ and $\beta_{.j}$,
respectively, and the former is larger than the latter.  Furthermore,
if $l'$ is the smallest position where $\alpha',\beta'$ differ, then
$\alpha_{., \pi(l')} \neq \beta_{., \pi(l')}$ and hence $\pi(l') \geq j$
and hence $l'=\min \pi^{-1}(\pi(l')) \geq \min \pi^{-1}(j) = l$. Hence
in fact $l=l'$ and $\pi^* \alpha > \pi^* \beta$.
\end{proof}

In addition to the well-order $\leq$ on each individual
$M_n(S)$, we also need the following partial order $|$ on
the union of all of them.

\begin{de}
Let $S,T$ be objects in $\OS$. We say that $\alpha \in M_n(T)$ {\em
divides} $\beta \in M_n(S)$ if there exist a $\pi \in \Hom_{OS}(S,T)$
and a $\gamma \in M_n(S)$ such that $\beta=\gamma+\pi^* \alpha$. In this
case, we write $\alpha | \beta$. \hfill $\clubsuit$
\end{de}

The key combinatorial property of the relation just defined is the
following.

\begin{prop} \label{prop:WQO}
The relation $|$ is a well-quasi-order, that is, for any
sequence $\alpha^{(1)} \in M_n(S_1), \alpha^{(2)} \in M_n(S_2),
\ldots$ there exist $i<j$ such that
$\alpha^{(i)}|\alpha^{(j)}$. 
\end{prop}

\begin{proof}
First, to each $\alpha \in M_n(S)$ we associate the monomial ideal
$P(\alpha)$ in the polynomial ring $R:=K[z_1,\ldots,z_n]$ (here $K$
is but a place holder) generated by the monomials $z^{\alpha_{.,j}},
j \in S$. The crucial fact that we will use about monomial ideals in $R$
is that in any sequence $P_1,P_2,\ldots$ of such ideals there exist $i<j$
such that $P_i \supseteq P_j$---in other words, monomial ideals are
well-quasi-ordered with respect to reverse inclusion \cite{MacLagan01}.

To prove the proposition, suppose, on the contrary, that there exists a sequence as above with
$\alpha^{(i)}\! \not\!|\  \alpha^{(j)}$ for all $i<j$. Such a sequence is
called {\em bad}. Then by basic properties of well-quasi-orders, some bad
sequence exists in which moreover
\begin{equation} \label{eq:J}
P(\alpha^{(1)}) \supseteq P(\alpha^{(2)}) \supseteq \ldots 
\end{equation} 
Among all bad sequences {\em with this additional property} choose one
in which, for each $j=1,2,\ldots,$ the cardinality $|S_j|$ is minimal among all bad sequences starting with
$\alpha^{(1)},\ldots,\alpha^{(j-1)}$.

Write $\alpha^{(j)}=(\gamma^{(j)}|\beta^{(j)})$, where $\beta^{(j)} \in
\ZZ_{\geq 0}^n$ is the last column (the one labelled by the largest
element of $S_j$), and $\gamma^{(j)}$ is the remainder.  By Dickson's lemma, there exists a subsequence
$j_1<j_2<\ldots$ such that $\beta^{(j_1)},\beta^{(j_2)},\ldots$ increase
weakly in the coordinate-wise ordering on $\ZZ_{\geq 0}^n$. By restricting
to a further subsequence, we may moreover assume that also
\begin{equation} \label{eq:Jg}
P(\gamma^{(j_1)}) \supseteq P(\gamma^{(j_2)}) \supseteq
\ldots 
\end{equation}
Then consider the new sequence
\[
\alpha^{(1)},\ldots,\alpha^{(j_1-1)},\gamma^{(j_1)},\gamma^{(j_2)},\ldots
\]
By~\eqref{eq:Jg} and~\eqref{eq:J}, and since $P(\alpha^{(j)}) \supseteq
P(\gamma^{(j)})$, this sequence also satisfies \eqref{eq:J}. We claim
that, furthermore, it is bad. 

Suppose, for instance, that $\gamma^{(j_1)} | \gamma^{(j_2)}$. Set
$a_i:=\max S_{j_i}$ for $i=1,2$. Then there exists a $\pi \in
\Hom_{\OS}(S_{j_2} \setminus \{a_2\},S_{j_1} \setminus \{a_1\})$
such that $\gamma^{(j_2)}-\pi^* \gamma^{(j_1)} \in M_n(S_{j_2}
\setminus \{a_2\})$. But then extend $\pi$ to an element $\pi$
of $\Hom_\OS(S_{j_2},S_{j_1})$ by setting $\pi(a_2):=a_1$; since
$\beta^{(j_1)}$ is coordinate-wise smaller than $\beta^{(j_2)}$
we find that $\alpha^{(j_2)}-\pi^* \alpha^{(j_1)} \in M_n(S_{j_2})$, so
$\alpha^{(j_1)}|\alpha^{(j_2)}$, in contradiction to the badness of
the original sequence.

On the other hand, suppose for instance that $\alpha^{(1)}
| \gamma^{(j_2)}$ and write $a_2:=\max S_{j_2}$.  Then there
exists a $\pi \in \Hom_{\OS}(S_{j_2} \setminus \{a_2\},S_1)$ such
that $\gamma^{(j_2)}-\pi^*(\alpha^{(1)}) \in M_n(S_{j_2} \setminus
\{a_2\})$. Now, {\em and this is why we required that~\eqref{eq:J} holds},
since $P(\alpha^{(j_2)}) \subseteq P(\alpha^{(1)})$, there exists an
element $s \in S_1$ such that the column $\beta^{(j_2)}$ is coordinatewise
at least as large as the $s$-th column of $\alpha^{(1)}$. Extend
$\pi$ to an element of $\Hom_\OS(S_{j_2},S_1)$ by setting $\pi(a_2) = s$. Since $a_2$ is the maximal element of $S_{j_2}$, this
does not destroy the property that the function $\min \pi^{-1}(.)$
be increasing in its argument. Moreover, this $\pi$ has the property
that $\alpha^{(j_2)}-\pi^*\alpha^{(1)} \in M_n(S_{j_2})$, again a
contradiction.

Since we have found a bad sequence satisfying \eqref{eq:J} but with
strictly smaller underlying set at the $j_1$-st position, we have arrived
at a contradiction. 
\end{proof}

Next we use a Gr\"obner basis argument.

\begin{proof}[Proof of Theorem~\ref{thm:RankOne}]
We prove the stronger statement that $KM_n$
is Noetherian as an $\OS$-algebra.  Let $P$ be any ideal in $A_n/I_n
= KM_n$. For each object $S$ in $\OS$, let $L(S) \subseteq M_n(S)$
denote the set of leading terms of nonzero elements of $P(S)$ relative
to the ordering $>$.  Proposition~\ref{prop:WQO} implies that there
exists a finite collection $S_1,\ldots,S_N$ and $\alpha^{(j)} \in
L(S_j)$ such that each element of each $L(S)$ is divisible by some
$\alpha^{(j)}$. Correspondingly, there exist elements $f_j \in P(S_j)$
with leading monomial $\alpha^{(j)}$ and leading coefficient $1$. To see
that the $f_j$ generate $P$, suppose that there exists an $S$ such that
$P(S)$ is not contained in the ideal generated by the $f_j$, and let $g
\in P(S)$ have minimal leading term $\beta$ among all elements of $P(S)$
not in the ideal generated by the $f_j$; without loss of generality $g$
has leading coefficient $1$.  By construction, there exists some $j$ and
some $\pi \in \Hom_{\OS}(S,S_j)$ such that $\beta-\pi^* \alpha^{(j)}
\in M_n(S)$. But now, by Lemma~\ref{lm:Order}, we find that the leading
monomial of $\pi^* f_j$ equals $\pi^* \alpha^{(j)}$, hence
subtracting a monomial times $\pi^* f_j$ from $g$ we obtain
an element of $P(S)$ with smaller leading monomial that is
not in the ideal generated by the $f_j$---a contradiction.
\end{proof}

Below, we need the following generalisation of
Theorem~\ref{thm:RankOne}.

\begin{thm} \label{thm:Tuples}
For any $n_1,\ldots,n_r \in (\ZZ)_{\geq 0}$ the $\Fin$-algebra (or
$\OS$-algebra) $(A_{n_1}/I_{n_1}) \otimes \cdots \otimes (A_n/I_{n_1})$
is Noetherian.
\end{thm}

\begin{proof}
This algebra is isomorphic to $B:=K (M_{n_1} \times \cdots \times M_{n_r})$. There
is a natural embedding $\iota: M_{n_1} \times \cdots \times M_{n_r} \to
M_{n_1+\ldots+n_s}=:M_n$ by forming a block matrix; its image consists
of block matrices with constant partial column sums.  And while
a subalgebra of a Noetherian algebra is not necessarily
Noetherian, this is true in the current setting. 

The crucial point is
that if $\alpha_i \in (M_{n_1} \times \cdots \times M_{n_r}) (S_i)$
for $i=1,2$, then {\em a priori} $\iota(\alpha_1) | \iota(\alpha_2)$
only means that $\iota(\alpha_2)-\pi^* \iota(\alpha_1) \in M_n(S_2)$;
but since both summands have constant partial column sums, so does their
difference, so in fact, the difference lies in the image of $\iota$. With
this observation, the proof above for the case where $r=1$ goes through 
unaltered for arbitrary $r$.
\end{proof}

\begin{re}
Similar arguments for passing to sub-algebras are also used in
\cite{Hillar09} and \cite{Draisma08b}. \hfill $\clubsuit$
\end{re}

\section{Proofs of the main results} \label{sec:Proof}

In this section we prove Theorems~\ref{cor:Markov} and~\ref{thm:Main}.

\subsection*{Toric fibre products}

To prove Theorem~\ref{thm:Main}, we work with a product of $s$ copies of the category $\Fin$; one for each of the varieties $X_i$ whose iterated
fibre products are under consideration.  Let $s,r \in \ZZ_{\geq 0}.$
For each $i \in [s]$ and $j \in [r]$ let $d_{ij} \in
\ZZ_{\geq 0}$ and set $V_{ij}:=K^{d_{ij}}$. Consider the $\Fin^s$-variety
$\cV$ that assigns to an $s$-tuple $S=(S_1,\ldots,S_s)$ the product
\[ \prod_{j=1}^r \bigotimes_{i=1}^s V_{ij}^{\otimes S_i} \]
and to a morphism $\pi=(\pi_1,\ldots,\pi_s):S
\to
T:=(T_1,\ldots,T_s)$ in $\Fin^s$ the linear map $\cV(S) \to \cV(T)$ determined by 
\begin{equation}  \label{eq:cV}
\left(\otimes_i \otimes_{k \in S_i} v_{ijk}\right)_{j \in [r]}
\mapsto 
\left(\otimes_i \otimes_{l \in T_i} \left(\bigcirc_{k \in \pi_i^{-1}(l)}
v_{ijk}\right)\right)_{j \in [r]}, 
\end{equation}
where the Hadamard product $\bigcirc$ is the one on $V_{ij}$.  Let
$\cQ^{\leq 1}(S)$ be the Zariski-closed subset of $\cV(S)$ consisting
of $r$-tuples of tensors of rank at most one; thus $\cQ^{\leq 1}$ is a
$\Fin^s$-subvariety of $\cV$.

For each $i \in [s]$ let $X_i \subseteq \prod_{j=1}^r V_{ij}$
be a Hadamard-stable Zariski-closed subset. Then for any tuple
$S=(S_1,\ldots,S_s)$ in $\Fin^s$ the variety $\cX(S):=X_1^{* S_1} *
\cdots * X_s^{* S_s}$ is a Zariski-closed subset of $\cV(S)$.

\ \\

\begin{lm} \label{lm:FinsSub}
The association $S \mapsto \cX(S)$ defines a $\Fin^s$-closed subvariety
of $\cQ^{\leq 1}$.
\end{lm}

\begin{proof}
From the definition of $*$ in \eqref{eq:Toric} it is clear that the
elements of $\cX(S)$ are $r$-tuples of tensors of rank at most $1$.
Furthermore, for a morphism $S \to T$ in $\Fin^s$ the linear map 
$\cV(S) \to \cV(T)$ from~\eqref{eq:cV} sends $\cX(S)$ into $\cX(T)$---here 
we use that if $(v_{i,j,k})_{j \in [r]} \in X_i$ for each $k \in \pi^{-1}(l)$, 
then also $(\bigcirc_{k \in \pi^{-1}(l)} v_{ijk})_{j \in [r]} \in X_i$ 
since $X_i$ is Hadamard-stable. 
\end{proof}

Now Theorem~\ref{thm:Main} follows once we know that the coordinate ring
of $\cQ^{\leq 1}$ is a Noetherian $(\Fin^s)^\op$-algebra.  For $s=1$ and
$r=1$ this is Theorem~\ref{thm:RankOne} with $n$ equal to $d_{11}$. For
$s=1$ and general $r$, this is Theorem~\ref{thm:Tuples} with $n_j$
equal to $d_{1j}$.

For $r=1$ and general $s$, Theorem~\ref{thm:Main} follows from a
$\Fin^s$-analogue of Theorem~\ref{thm:RankOne}, which is proved as
follows. The coordinate ring of $\cQ^{\leq 1}(S_1,\ldots,S_s)$ is the
subring of $K (M_{d_{11}}(S_1) \times \cdots \times M_{d_{s1}}(S_s))$
spanned by the monomials corresponding to $s$-tuples of matrices with
the {\em same} constant column sum. Using Proposition~\ref{prop:WQO}
and the fact that a finite product of well-quasi-ordered sets is
well-quasi-ordered one finds that the natural $\Fin^s$-analogue on
$M_{d_{11}} \times \cdots \times M_{d_{s1}}$ of the division relation $|$
is a well-quasi-order; and this implies, once again, that the coordinate
ring of $\cQ^{\leq 1}$ is a Noetherian $(\OS^s)^\op$-algebra.

Finally, for general $s$ and general $r$, the result follows as in the
proof of Theorem~\ref{thm:Tuples}. This proves the Theorem~\ref{thm:Main}
in full generality. \hfill $\square$

\begin{re} \label{re:NoOne}
The only place where we used that the $X_i$ contain the all-one
vector is in the proof of Lemma~\ref{lm:FinsSub} when $\pi^{-1}(l)$
happens to be empty. If we do not require this, then the conclusion of
Theorem~\ref{thm:Main} still holds, since one can work directly with
the category $\OS^s$ in which morphisms $\pi$ are surjective.
\end{re}

\begin{re} \label{re:Scheme}
If we replace the $X_i$ by Hadamard-stable closed subschemes rather than
subvarieties, then $S \mapsto \cX(S)$ is still a $\Fin^s$-closed
subscheme of $\cQ^{\leq 1}$, and since the coordinate ring of the latter
is Noetherian, the proof goes through unaltered.
\end{re}

\begin{re} \label{re:Independent}
In the Independent Set Theorem from~\cite{Hillar09}, the graph $G=(N,E)$
is fixed but the state space sizes $d_j$ grow unboundedly for $j$ in
an independent set $T \subseteq N$ and are fixed for $j \not \in T$. In
this case, given a $T$-tuple of maps $(\pi_j:S_j \to P_j)_{j \in T}$ of
finite sets, where $S_j$ is thought of as the state space of $j \in T$
in the smaller model and $P_j$ as the state space in the larger model, we
obtain a natural map from the larger model into the smaller model. Hence
then the graphical model is naturally a $(\Fin^\op)^T$-variety and its
coordinate ring is a $\Fin^T$-algebra. Note the reversal of the roles
of these two categories compared to Lemma~\ref{lm:FinsSub}.
\end{re}

\subsection*{Markov random fields} \label{sec:Markov}

Given a finite (undirected, simple) graph $G=(N,E)$ with a
number $d_j$ of states attached to each node $j \in N$, the
graphical model is $X_G:=\overline{\im \phi_G}$, where
$\phi_G$ is the parameterisation
\[
\phi_G: \CC^{\prod_{C\in\maxcl(G)}\prod_{j\in
C}[d_j]}\to\CC^{\prod_{j\in N}[d_j]}, \qquad
(\theta^C_\alpha)_{C,\alpha}\mapsto
\left(\prod_{C\in\maxcl(G)}\theta^C_{\beta|_C}\right)_{\beta}.
\]

\begin{lm} \label{lm:GraphHadamard}
For any finite graph $G$, the graphical model $X_G$ is
Hadamard-closed.
\end{lm}

\begin{proof}
The parameterisation $\phi$ sends the all-one vector in the domain
to the all-one vector $\one$ in the target space, so $\one \in \im \phi$. Moreover, if $\theta,\theta'$ are two parameter vectors, then
$\phi(\theta \bigcirc \theta')=\phi(\theta) \bigcirc \phi(\theta')$,
so $\im \phi$ is Hadamard-closed. Then so is its closure. 
\end{proof}

Following~\cite{Sullivant07}, we relate graph glueing to toric fibre
products. We are given finite graphs $G_1,\ldots,G_s$ with node sets
$N_1,\ldots,N_s$ such that $N_i \cap N_k=N_0$ for all $i \neq k$ in $[s]$
and such that each $G_i$ induces the same graph $H$ on $N_0$. Moreover,
for each $j \in \bigcup_i N_i$ we fix a number $d_j$ of states.

For each $\beta_0 \in \prod_{j \in N_0} [d_j]$ and each $i \in [s]$
set $V_{i,\beta_0}:=\CC^{\prod_{j \in N_i \setminus N_0}[d_j]}$, which
we interpret as the ambient space of the part of the probability table
of the variables $X_j, j \in N_i$ where we have fixed the states of
the variables in $N_0$ to $\beta_0$---up to scaling, these are the
conditional joint probabilities for the $X_j, j \in N_i \setminus N_0$
given that the $X_j, j \in N_0$ are in joint state $\beta_0$.
For $\beta \in \prod_{j \in N_i} [d_j]$ write $\beta=\beta_0||\beta'$
where $\beta_0,\beta'$ are the restrictions of $\beta$ to $N_0$ and $N_i
\setminus N_0$, respectively.  For each maximal clique $C$ in $G_i$ define
$C_0=C \cap N_0$ and $C'=C \setminus N_0$. Correspondingly, decompose
$\alpha \in \prod_{j \in C} [d_j]$ as $\alpha=\alpha_0||\alpha'$, where
$\alpha_0,\alpha'$ are the restrictions of $\alpha$ to $C_0$ and $C'$,
respectively.

Then the graphical model $X_{G_i}:=\overline{\im
\phi_{G_i}}$ is the closure of the image of the parameterisation
\[ \phi_i:\CC^{\prod_{C \in \maxcl(G_i)} \prod_{j \in C} [d_j]} \to 
\prod_{\beta_0 \in \prod_{j \in N_0} [d_j]} V_{i,\beta_0}, 
\qquad 
(\theta_{\alpha}^C)_{C,\alpha} \mapsto 
\left( \left(\prod_{C \in \maxcl(G)}
\theta^C_{(\beta_0|_{C_0})\ ||\ (\beta'|_{C'})}
\right)_{\beta'} \right)_{\beta_0}.
\]
Setting $r:=\prod_{j \in N_0} d_j$, we are exactly in the setting of the
previous sections: for each $i,k \in [s]$, we have the bilinear map 
\[ 
*:\prod_{\beta_0} V_{i,\beta_0} \times \prod_{\beta_0} V_{k,\beta_0}
\to \prod_{\beta_0} (V_{i,\beta_0} \otimes
V_{k,\beta_0}),\qquad
((v_{i,\beta_0})_{\beta_0},(v_{k,\beta_0})_{\beta_0})
\mapsto 
(v_{i,\beta_0} \otimes v_{k,\beta_0})_{\beta_0}, 
\] 
and we can take iterated products of this type. The space on the
right is naturally isomorphic to $\CC^{\prod_{j \in N_1 \cup
N_2} [d_j]}$, the space of probability tables for the joint
distribution of the variables labelled by the vertices in
the glued graph $G_i +_H G_k$. Under this identification we
have the following.

\begin{prop} \label{prop:GlueToric}
For $G:=\sum_H^{a_1} G_1 +_H \cdots +_H \sum_H^{a_s} G_s$ we
have $X_G =X_{G_1}^{* a_1} * \cdots *
X_{G_s}^{* a_s}$.
\end{prop}

\begin{proof}
It suffices to prove this for the gluing of two graphs. Note that
a clique in $G:=G_1 +_H G_2$ is contained entirely in either $G_1$
or $G_2$, or in both but then already in $H$.  Let $\theta,\eta$ be a
parameter vectors in the domains of $\phi_{G_1},
\phi_{G_2}$, respectively. Then 
\begin{align*}
&\phi_{G_1}(\theta) *
\phi_{G_2}(\eta) \\
 &= 
\left( \left(\prod_{C \in \maxcl(G_1)}
\theta^{C}_{(\beta_0|_{C_0})\ ||\ (\beta'|_{C'})}
\right)_{\beta'} \right)_{\beta_0}
*
\left( \left(\prod_{C \in \maxcl(G_2)}
\eta^{C}_{(\beta_0|_{C_0})\ ||\ (\beta'|_{C'})}
\right)_{\beta'} \right)_{\beta_0}\\
 &= 
\left( \left(\prod_{C \in \maxcl(G_1)}
\theta^{C}_{(\beta_0|_{C_0})\ ||\ (\beta'|_{C'})}
\right)_{\beta' \in \prod_{j \in N_1 \setminus N_0}[d_j]} 
\otimes 
\left(\prod_{C \in \maxcl(G_2)}
\eta^{C}_{(\beta_0|_{C_0})\ ||\ (\beta'|_{C'})}
\right)_{\beta' \in \prod_{j \in N_2 \setminus N_0}[d_j]}
\right)_{\beta_0}\\
 &= 
\left( \left(\prod_{C \in \maxcl(G_1)}
\theta^{C}_{(\beta_0|_{C_0})\ ||\ (\beta'|_{C'})}
\cdot
\prod_{C \in \maxcl(G_2)}
\eta^{C}_{(\beta_0|_{C_0})\ ||\ (\beta'|_{C'})}
\right)_{\beta' \in \prod_{j \in (N_1 \cup N_2) \setminus
N_0} [d_j]} \right)_{\beta_0}\\
&= 
\left( \left(\prod_{C \in \maxcl(G)}
\mu^{C}_{(\beta_0|_{C_0})\ ||\ (\beta'|_{C'})}
\right)_{\beta' \in \prod_{j \in (N_1 \cup N_2) \setminus
N_0} [d_j]} \right)_{\beta_0}=\phi_G(\mu),
\end{align*}
where, for $C \in \maxcl(G)$ and $\alpha \in \prod_{j \in
C}[d_j]$, the parameter $\mu^C_\alpha$ is defined as
\[ \mu^C_\alpha:=
\begin{cases}
\theta^C_\alpha & \text{ if $C \subseteq N_1$ and $C\not\subseteq N_0$,}\\
\eta^C_\alpha & \text{ if $C \subseteq N_2$ and $C\not\subseteq N_0$, and}\\
\theta^C_\alpha \ \eta^C_\alpha & \text{ if $C \subseteq N_0$.}
\end{cases}
\]
This computation proves that $X_{G_1} * X_{G_2} \subseteq
X_{G}$. Conversely, given any parameter vector $\mu$ for $G$, we can let
$\theta$ be the restriction of $\mu$ to maximal cliques of the first
and third type above, and set $\eta^C_\alpha$ equal to $\mu^C_\alpha$
if $C$ is of the second type above and equal to $1$ if it is of the
third type. This yields the opposite inclusion.
\end{proof}

\begin{proof}[Proof of Theorem~\ref{cor:Markov}.]
By Proposition~\ref{prop:GlueToric}, the ideal $I_G$ is the ideal
of the iterated toric fibre product $X_{G_1}^{* a_1} * \cdots *
X_{G_s}^{*a_s}$. By Lemma~\ref{lm:GraphHadamard}, each of the varieties
$X_{G_i}$ is Hadamard closed. Hence Theorem~\ref{thm:Main} applies,
and $I_G$ is generated by polynomials of degree less than some $D$,
which is independent of $a_1,\ldots,a_s$. Then it is also
generated by the binomials of at most degree $D$.
\end{proof}


\begin{thebibliography}{Sam17b}

\bibitem[DK14]{Draisma11d}
Jan Draisma and Jochen Kuttler.
\newblock Bounded-rank tensors are defined in bounded degree.
\newblock {\em Duke Math.~J.}, 163(1):35--63, 2014.

\bibitem[Dra10]{Draisma08b}
Jan Draisma.
\newblock Finiteness for the k-factor model and chirality varieties.
\newblock {\em Adv.~Math.}, 223:243--256, 2010.

\bibitem[DS98]{Diaconis98}
Persi Diaconis and Bernd Sturmfels.
\newblock Algebraic algorithms for sampling from conditional distributions.
\newblock {\em Ann. Stat.}, 26(1):363--397, 1998.

\bibitem[EKS14]{Engstrom14}
Alexander {Engstr\"om}, Thomas {Kahle}, and Seth {Sullivant}.
\newblock {Multigraded commutative algebra of graph decompositions.}
\newblock {\em {J. Algebr. Comb.}}, 39(2):335--372, 2014.

\bibitem[ES96]{Eisenbud96}
David {Eisenbud} and Bernd {Sturmfels}.
\newblock {Binomial ideals.}
\newblock {\em {Duke Math. J.}}, 84(1):1--45, 1996.

\bibitem[HC71]{Hammersley71}
J.M. Hammersley and P.~Clifford.
\newblock Markov fields on finite graphs and lattices.
\newblock Unpublished,
  \verb+http://www.statslab.cam.ac.uk/~grg/books/hammfest/hamm-cliff.pdf+,
  1971.

\bibitem[HS12]{Hillar09}
Christopher~J. Hillar and Seth Sullivant.
\newblock Finite {G}r\"obner bases in infinite dimensional polynomial rings and
  applications.
\newblock {\em Adv. Math.}, 221:1--25, 2012.

\bibitem[KR14]{Kahle14}
Thomas {Kahle} and Johannes {Rauh}.
\newblock {Toric fiber products versus Segre products.}
\newblock {\em {Abh. Math. Semin. Univ. Hamb.}}, 84(2):187--201, 2014.

\bibitem[Lau98]{Lauritzen98}
Steffen~L. Lauritzen.
\newblock {\em Graphical models}, volume~17 of {\em Oxford Statistical Science
  Series}.
\newblock Oxford Univ. Press., Oxford, 1998.

\bibitem[Mac01]{MacLagan01}
Diane MacLagan.
\newblock Antichains of monomial ideals are finite.
\newblock {\em Proc. Am. Math. Soc.}, 129(6):1609--1615, 2001.

\bibitem[Oos16]{Oosterhof16}
Florian~M. Oosterhof.
\newblock {\em Stabilisation of iterated toric fibre products}.
\newblock 2016.
\newblock Master's thesis, Eindhoven University of Technology,
  \verb+http://repository.tue.nl/855107+.

\bibitem[RS14]{Rauh14b}
Johannes Rauh and Seth Sullivant.
\newblock The {M}arkov basis of $k_{3,N}$.
\newblock \verb+arXiv:1406.5936+, 2014.

\bibitem[RS16]{Rauh14}
Johannes Rauh and Seth Sullivant.
\newblock Lifting {M}arkov bases and higher codimension toric fiber products.
\newblock {\em J.~Symb.~Comp}, 74:276--307, 2016.

\bibitem[Sam17a]{Sam15b}
Steven~V. Sam.
\newblock Ideals of bounded rank symmetric tensors are generated in bounded
  degree.
\newblock {\em Invent.~Math.}, 207(1):1--21, 2017.

\bibitem[Sam17b]{Sam16}
Steven~V. Sam.
\newblock Syzygies of bounded rank symmetric tensors are generated in bounded
  degree.
\newblock {\em Math. Ann.}, 368(3--4):1095--1108, 2017.

\bibitem[{Shi}12]{Shibuta12}
Takafumi {Shibuta}.
\newblock {Gr\"obner bases of contraction ideals.}
\newblock {\em {J. Algebr. Comb.}}, 36(1):1--19, 2012.

\bibitem[SS17]{Sam14}
{Steven V.~Sam} and Andrew Snowden.
\newblock Gr\"obner methods for representations of combinatorial categories.
\newblock {\em J.~Am.~Math.~Soc.}, 30(1):159--203, 2017.

\bibitem[{Sul}07]{Sullivant07}
Seth {Sullivant}.
\newblock {Toric fiber products.}
\newblock {\em {J. Algebra}}, 316(2):560--577, 2007.

\end{thebibliography}

\end{document}